\documentclass[11pt]{amsart}
\usepackage{amsfonts, amstext, amsmath, amsthm, amscd, amssymb}
\usepackage{epsfig, graphics, psfrag}

\address[kalfagia@math.msu.edu]{E. Kalfagianni
Mathematics Department, Michigan State
  University, E. Lansing, MI 48824, USA}

\newtheorem{theorem}{Theorem}[section]

\newtheorem{lemma}[theorem]{Lemma}

\theoremstyle{definition}

\newtheorem{remark}[theorem]{Remark}
\newtheorem{define}{Definition}[section]
\newtheorem*{question}{Question}

\newtheorem{prop}[theorem]{Proposition}

\DeclareMathOperator{\Hom}{Hom}

\def\arrowdown#1#2{\Big\downarrow \rlap{$\vcenter{\hbox{$\scriptstyle#2$}}$}
{\hbox to -10pt{\hss{$\vcenter{\hbox{$\scriptstyle#1$}}$}}}}
 
\newcommand{\Z}{{\mathbb{Z}}}
\newcommand{\N}{{\mathbb{N}}}
\newcommand{\Q}{{\mathbb{Q}}}

\title[Invariants of framed links in 3-manifolds]{An intrinsic approach to invariants of framed links in 3-manifolds}

\author[E.\ Kalfagianni]{Efstratia Kalfagianni}
\thanks{Supported in part by  NSF grant DMS-0805942}

\begin{document}

\begin{abstract} We study framed links in irreducible  3-manifolds that are $\Z$-homology 3-spheres or atoroidal $\Q$-homology 3-spheres.
We calculate the dual of the Kauffman skein module over the ring of two variable power series with complex coefficients.
For links in $S^3$ we give a new construction of the classical Kauffman
polynomial. 

\vskip 0.3in

\vskip 0.1in

\noindent {{\bf Mathematics Subject ClassiÞcation (2010).} } 57N10, 57M2, 57R42, 57R56.
\end{abstract}

\maketitle

\section{Introduction \label{sec:intro}}
The Kauffman polynomial is a 2-variable Laurent polynomial invariant
for links in $S^3$ \cite{kau} that has interesting applications and connections with contact geometry. 
The  degree in one of the variables of the Kauffman polynomial provides an upper bound for the Thurston-Bennequin norm of Legendrian links \cite{fer, taba}. The inequality is known to be sharp for several classes of links (e.g. alternating links) and the proof of this sharpness has led to deeper connections between knot polynomials and contact geometry \cite{rath}.
\vskip 0.03in

In this paper we study framed links in oriented, irreducible 3-manifolds that are $\Z$-homology 3-spheres or atoroidal $\Q$-homology 3-spheres. We give conditions under which an invariant that is defined on framed singular links with one 
double point gives rise to an invariant of framed links (Theorem \ref{integration}). 
This allows us to construct formal power series framed link invariants
obeying the Kauffman polynomial skein
relations. The coefficients of these series are finite type framed link invariants
and are perturbative versions of the Reshetikhin-Turaev, Witten  $SO(n)$-invariants 
\cite{rt, witten} in the sense of Le-Murakami-Ohtsuki  \cite{LMO}.
Using weight systems corresponding to appropriate representations of the Lie algebras $so(n)$
and the naturallity of the LMO invariant, one obtains a Kauffman type power series invariant for framed links in all
$\Q$-homology 3-spheres.
Our approach in this paper 
is quite different from this line and allows us to solve the subtler problem of constructing
power series  invariants with given values on a set of initial links.
Our approach here, that exhibits the interplay 
between skein framed link  theory and the  topology of 3-manifolds, is inspired by the study of  Vassiliev invariants (a.k.a. finite type  invariants) \cite{vasi}
using
3-dimensional topology techniques  \cite{ka}.  The precise relation of the power series constructed here to the one obtained via the LMO invariant
is not clear to us at this point.

\begin{define} \label{framedlinks} Let $M$ be an oriented $\Q$-homology 3-sphere. A {\em framed $m$-component link} is a collection of  $m$ unordered (unoriented) circles smoothly and disjointly embedded in $M$ and such that each component  is equipped with 
a continuous unit normal vector field. Two framed links are equivalent if they are isotopic 
by an ambient isotopy that preserves the homotopy class of the vector field on each component.
Let  ${\mathcal {\bar L}}:={\mathcal {\bar L}}(M)$ denote the set of isotopy classes of framed links in $M$.
\end{define}
\begin{figure}[htbp] %
\centering
\includegraphics[width=3.2in, ]{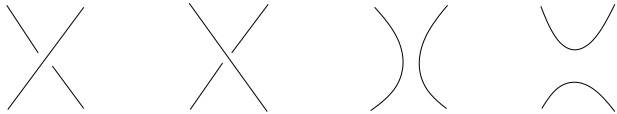} 
\caption{The parts  of $L_+$, $L_-$ and $L_o$ and $L_{\infty}$ in $B$. }
\end{figure}
\begin{figure}[htbp]
\centering
   \includegraphics[width=3.3in, ]{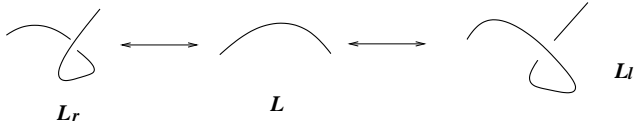} 
   
 \caption{$L_r$ and $L_l$ are obtained by a full twist from $L$. }
\end{figure}

To state the main result of the paper we need some notation and conventions:
Let  $L_+$, $L_-$, $L_o$ and  $L_{\infty}$  denote four framed links that are identical everywhere
except in a 3-ball $B$ in $M$. There under a
suitable projection of the parts in $B$,  $L_+$, $L_-$, $L_o$ and $L_{\infty}$ look as shown in Figure 1.  
Also for every framed link we denote by $L_r, L_l$ the framed links that are identical to $L$ everywhere except in a 3-ball where
they differ as shown in Figure 2.  Here we suppose that the orientation of $M$ agrees with the right-handed 
orientation of the 3-balls containing the link parts in Figures 1 and 2
and that 
the framing vector for link parts in these figures  is perpendicular to the page.
The framings of the links coincide everywhere
 outside
the parts shown in Figures 1 and 2.

Let 
$\hat \Lambda:={ \mathbb C}[[x, \ y]]$ denote the ring of formal power series in  $x,y$ over $\mathbb C$ and let ${\displaystyle t:={e}^x=1+x+{{x^2}\over {2}}+  \dots}$. 
Let us set
${\displaystyle a:= i{e}^{y}=i+iy+{i{y^2}\over {2}}+ \dots}$ and set
${\displaystyle z:= it-{(it)^{-1}}=i{e}^x+ie^{-x}= 2i+{{ix^2}}+\dots}$.
Note that  $a$ and  $z$  are invertible  in $\hat \Lambda$.
\begin{define} \label{skein}The Kauffman {\em skein module} of $M$ over $\hat \Lambda$, denoted by ${\mathfrak F} (M)$, is the quotient of the free $\hat \Lambda$-module with basis
${\mathcal {\bar  L}}$ by its ideal generated by all the relations of the following two types:
$$L_+-L_-=z\big[ L_o-L_\infty\big],$$
$$L_r=a L\  \  {\rm and} \  \  L_l=a^{-1} L.$$
 We will use ${\mathfrak F}^{*}(M):= {\Hom}_{\hat \Lambda} {\big( {\mathfrak F}(M), \hat \Lambda \big)}$ 
to denote the  $\hat \Lambda $-dual of ${\mathfrak F}(M)$.
\end{define}

\begin{remark}{\rm The usual convention in skein module theory is to allow an empty link as part of the set ${\mathcal {\bar L}}$.
In contrast to that, in this paper, we find it convenient to work with non-empty links (Definition \ref{framedlinks}).}
\end{remark}
\begin{remark} Since the links are unoriented 
the declarations $L_+$ and $L_-$, when considering a crossing,  are arbitrary. However this doesn't matter for our purposes since the first skein relation in Definition
\ref{skein} is invariant under simultaneously interchanging $L_+$ with $L_-$ and $L_o$ with $L_{\infty}$.
\end{remark}

To continue let $\pi:= {\pi}(M)$ denote the set of  non-trivial conjugacy classes of $\pi_1(M)$  and
${\hat \pi}$ denote the set obtained from $\pi$ by indentifying
the conjugacy class of every element  $1\neq x\in \pi_1(M)$ with that of $x^{-1}$.
Also let
$S( {\hat \pi})$ denote the symmetric algebra of the free $\hat \Lambda$-module, say  $\hat \Lambda \hat \pi$, with basis $\hat \pi$.
Finally, let
$S^{*}({\hat \pi}):= {\Hom}_{\hat \Lambda}\big( S({\hat \pi}),  \hat \Lambda \big)$
denote the $\hat \Lambda $-dual of $S( {\hat \pi})$.

\begin{theorem}\label{main} Let $M$ be an oriented
$\Q$-homology 3-sphere with  
$\pi_2(M)=0$ and such  that if   $H_1(M)\neq 0$ then  $M$ is atoroidal. 
Then there is  a $\hat \Lambda$-module isomorphism
$${\mathfrak F}^{*}(M) \cong S^{*}({\hat \pi}).$$
\end{theorem}

For components that are homologically trivial in $M$ 
the homotopy class of the framing vector field is determined by an integer: the algebraic intersection number of
a push-out of the component in the direction of the framing vector field with a Seifert surface bounded by the component.
This algebraic intersection number is the self-linking number of the component.
There is a canonical framing defined by the Seifert surface that corresponds to the integer zero.
This implies that in  a $\Z$-homology sphere, for every underlying (unframed) isotopy class of knots the framed knot types correspond to integers.
The self-linking number can also be defined
in terms of Vassiliev-Gusarov axioms; it is a finite type framed link invariant of order one.  As shown by Chernov \cite{chernov} this point of view generalizes
to all framed knots in 3-manifolds; in particular for knots in irreducible $\Q$-homology 3-spheres that we study here.
For $M$ as above,  given a conjugacy class $c$  in $\pi_1(M)$ and a fixed framed knot  $CK$ representing $c$,
Chernov shows that there is a unique $\mathbb Z$-valued invariant for all framed knots representing $c$
with given value on $CK$ (Theorem 2.2  of \cite{chernov}).
His work implies that, with a chosen set of initial knots, for every underlying (unframed) isotopy class of knots
the framed knot types correspond to integers. This point will be useful to us in the next sections.
\vskip 0.06in

The isomorphism in Theorem \ref{main} also depends on a choice of initial links which we now discuss:
For every unordered sequence of elements in $\hat\pi\cup\{1\}$ we choose a framed link $CL$ that realizes it  and call it an
{\em initial link}. For  elements in $\hat\pi\cup\{1\}$ that are trivial in $H_1(M)$ we choose the canonical framing. This means that the integer describing the framing on each component
of an initial link is zero. For an initial link $CL$ with $k$ homotopically
trivial components we choose $CL=CL^*\sqcup U^k$, where 
$CL^* $ is an initial link with no homotopically trivial components and $U^k$
is the standard unlink in a 3-ball 
disjoint from $CL^*$. The one component unlink $U^1$ will be abbreviated to $U$. 
In general we will assume that each component of an initial link $CL$ is the chosen initial knot for the corresponding element in $\hat\pi\cup\{1\}$. 
We will also assume that each component is the initial knot required to define Chernov's self-linking invariant.
We will denote by $\mathcal C \mathcal L^*$ the set of all
initial 
links with no homotopically trivial components. 

The elements in the  set  $\mathcal C \mathcal L^* \cup \{U\}$ are in one-to-one correspondence with a basis of $S({\hat \pi})$.
An element  $R_M\in {\mathfrak F}^{*}(M)$ gives rise to one in $S^{*}({\hat \pi})$ by restriction on the set 
$\mathcal C \mathcal L^* \cup \{U\}$. Theorem \ref{main} will follow easily once we have proven the following result (see Section 4 for details).
\begin{theorem}\label{main2} Let $M$ be an oriented
$\Q$-homology 3-sphere with  
$\pi_2(M)=0$, and such that if   $H_1(M)\neq 0$ then $M$ is atoroidal. Given
a map ${\mathcal R}_M: \mathcal C \mathcal L^* \cup \{U\} \rightarrow\hat \Lambda$ there exists a unique
map $R_M: {\bar {\mathcal L}}\rightarrow\hat \Lambda$ such that:
\begin{enumerate}
\item The  restriction of $R_M$ on ${\mathcal C{ {\mathcal L^*}}}\cup \{ U\}$ is equal to ${\mathcal R}_M$.
\item $ R_M$ satisfies the Kauffman
skein relation
$$R_M(L_+)-R_M(L_-)=z\big[ R_M(L_o)-R_M(L_\infty)\big],$$
for every skein quadruple  of links $L_+$, $L_-$, $L_o$ and $L_{\infty}$ as in Figure 1.
 \item $R_M(L_r)=aR_M(L)$ and $R_M(L_l)=a^{-1}R_M(L)$ for every $L\in {\bar{\mathcal L}}$.
 \end{enumerate}
\end{theorem}

Let $\Lambda:={\mathbb C} [a^{\pm1},z^{\pm1}]$ denote the ring of Laurent polynomials in $a$
and $z$.  
We can define the Kauffman {\em skein module} of $M$ over $ \Lambda$, denoted by ${\mathfrak F}_{ \Lambda} (M)$, 
and consider its $ \Lambda$-dual, ${\mathfrak F}_{\Lambda}^{*}(M):= {\Hom}_{ \Lambda} {\big( {\mathfrak F}_{\Lambda}(M),  \Lambda \big)}$.
As we will discuss in Section 4, for links in $S^3$,
if we choose the value $R_{S^3}(U)$ to lie in $\Lambda$
then $R_{S^3}(L)\in \Lambda$, for every $L\in{\bar {\mathcal L}}$. This implies that 
 ${\mathfrak F}_{\Lambda}^{*}(S^3)\cong \Lambda$ and leads to the following question:

\begin{question} \label{convergence} Let $M$ be as in Theorem \ref{main}. Can we choose the initial links  $CL^*\in\mathcal C\mathcal L^*$
so that we have a $\Lambda$-module isomorphism
$${\mathfrak F}_{\Lambda}^{*}(M) \cong S_{\Lambda}^{*}({\hat \pi})?$$
Here, $S_{\Lambda}^{*}({\hat \pi})$ denotes the $ \Lambda$-dual of 
the symmetric algebra of the free $\Lambda$-module with basis $\hat \pi$.
\end{question}

In \cite{kalin} we constructed formal power series invariants that satisfy the HOMFLY skein change formula
for unframed oriented  links in large classes of $\Q$-homology 3-spheres. Cornwell  \cite{cr1, cr2, cr3} shows that for  lens spaces both the question above
and its analogue for the HOMFLY skein module of  \cite{ka2} have a positive answer.  
As a result he obtains analogues of the aforementioned results of 
\cite{fer, taba} for Legendrian links in contact lens spaces.
\vskip 0.04in

Theorem \ref{integration} of this paper is the framed link analogue of the ``integrability of singular link invariants"
results proved in \cite{ka,kalin}. Theorem  \ref{integration}  
doesn't follow from the results in these papers: In \cite{ka} we only treat knots
while in \cite{kalin} we treat links in some classes of irreducible $\mathbb Z$-homology 3-spheres.
In this paper we are able to  remove those restrictions and deal with all irreducible 
$\mathbb Z$-homology 3-spheres; see Theorem \ref{integration2} and Remark \ref{necessary}. 
If one forgets the framing, Theorem \ref{integration2}
generalizes the integrability results and Theorem A
of \cite{kalin} for links in all irreducible $\mathbb Z$-homology 3-spheres.
\vskip 0.04in

Framed links in general 3-manifolds and their skein modules were studied by several authors before; see \cite{jozef1} and references therein. 
In particular,
Przytycki \cite{jozef} introduced a two term homotopy skein module of framed links in oriented 3-manifolds as quantum deformation of 
the fundamental group.
In \cite{kaiser} Kaiser  calculated
this module over the ring of Laurent polynomials with $\Z$-coefficients.
He showed that if a 3-manifold contains no non-separating 2-spheres or tori
then Przytycki's module is a symmetric algebra of the free module with basis the set of non-trivial conjugacy
classes of $\pi_1(M)$.
Kaiser also studied several variations of two term skein modules and put
the classical self-linking number for null homologous knots as well as  Chernov's 
generalization of it in the skein module theory framework. For details the reader is referred to \cite{kaiser}.

The paper is organized as follows: In Section 2 we formulate the problem of integrating framed singular link invariants to invariants of framed links. Then we state an  integrability theorem and prove it for atoroidal $\Q$-homology spheres. In Section 3 we treat manifolds containing essential tori and in Section 4 we construct the Kauffman power series 
invariants and prove Theorems \ref{main} and \ref{main2}.

Throughout the paper we will work in the smooth category. 
\medskip

{{\bf Acknowledgment:}} I thank Chris Cornwell for his interest in this work and for several stimulating questions about link theory
in 3-manifolds that motivated me to go back and work on this project. I thank Vladimir Turaev for suggesting that I formulate the  main result of the paper in terms of skein modules. I am grateful to 
the anonymous referees for reading the paper carefully and making thoughtful comments and suggestions
that helped me improve the exposition.

\section{ Framed oriented Singular Link Invariants} 
Throughout this section we will work with oriented links  in oriented 3-manifolds. Theorem \ref{integration}, as well as its unframed counterparts \cite{ka, kalin, lin}, are proved for oriented links in oriented 3-manifolds.
For example,  the definitions of the signs of resolutions of double points below use 
the orientation of links as well as that of the ambient 3-manifolds. 

\subsection{Framed oriented singular links and resolutions}   Let $M$ be an oriented $\Q$-homology 3-sphere.
An $m$-component oriented {\em framed singular link} of order $n$ is a collection of unordered
oriented circles, smoothly immersed in $M$ such that (i) the only
singularities are
exactly $n$ transverse double points; and (ii)
the image of each component is equipped with 
a continuous unit normal vector field.  
We consider  framed singular links up to 
ambient isotopy that preserves the orientations, the  transversality of the double points and the homotopy class of the vector field on each component.
For $n=0$ we have an oriented framed link. We will denote by ${\mathcal L}^{(n)}:={\mathcal L}^{(n)}(M)$ (resp. ${\mathcal L}:={\mathcal L}(M)$) the set of isotopy classes of  oriented framed singular links of order $n$ (resp. links)
in $M$.

\smallskip
{\underline {\bf Convention:}} To simplify the exposition, for the remaining of the section and the next section, we will say  a framed link (resp. singular link) to mean an oriented framed link (resp. singular link). Also when we say a 3-manifold, we will mean an oriented 3-manifold.
\smallskip

Let $P$ denote a disjoint union of oriented circles and
consider a framed singular link represented by a smooth immersion
$L: P\longrightarrow M$.
Let $p\in M$ be a double point of $L$;
the inverse image consists of two points $p_1, p_2\in P$.
There are
disjoint intervals $\sigma_1$ and $\sigma_2$ on $P$ with $p_i
\in\hbox{int}(\sigma_i)$, $i=1, \ 2$, such that for a neighborhood $B$ of $p$ we have
$ L \cap B= L(\sigma_1)\cup L(\sigma_2)$.
Moreover, there is a proper 2-disc $D$ in $B$ such that $L(\sigma_1)$, $L(\sigma_2) \subset D$ intersect transversally at $p$. 
Now $L(\sigma_1)\cup L(\sigma_2)$ intersects $\partial D$
at four points  and, since $\sigma_i$
inherits an orientation from that of $P$, we can talk of the initial 
and terminal point of  $L(\sigma_i)$. Choose arcs $a_1$, $a_2$,
$b_1$, $b_2$ with disjoint interiors such that
\begin{enumerate}
 \item  $a_1$ and $a_2$ go from the initial point of $L(\sigma_1)$
to the terminal point of $L(\sigma_1)$ and lie in distinct components
of $\partial B \setminus \partial D$; and
 
 \item  $b_1$ and $b_2$ lie on $\partial D$ with $b_1$ going from
the initial point of $L(\sigma_1)$
to the terminal point of $L(\sigma_2)$ and $b_2$
from the initial point of $L(\sigma_2)$ to the terminal point of
$L(\sigma_1)$.  The complement of $b_1\sqcup b_2$
in $\partial D$ consists of two arcs, say $c_1, c_2$.
\end{enumerate}
The orientation of $M$ and that of $L(\sigma_2)$ define an orientation of $a_1\sqcup a_2$;
suppose that this induced orientation agrees with the one of $a_1$ and is opposite to that of $a_2$.
Define the positive resolution of $L$ at $p$ to be 
$$L_+=\overline {L\setminus L(\sigma_2)} \cup a_1,$$
and the negative resolution to be 
$$L_-=\overline {L\setminus L(\sigma_2)} \cup a_2.$$
In the case that $n=1$ we also define
$$L_o=\overline {L\setminus ( L(\sigma_2)\cup L(\sigma_1) )}\cup (b_1\sqcup b_2)$$
$$L_{\infty}=\overline {L\setminus ( L(\sigma_2)\cup L(\sigma_1) )}\cup (c_1\sqcup c_2)$$
Note
that $L_{\infty}$
only makes sense as an unoriented link.

\begin{define} \label{admissible} A framed singular link $L$  is called {\em inadmissible}
if  there is a 2-disc $D\subset M$
such that $L\cap D=\partial D$ and exactly  one double point
of $L$ lies on $\partial D$. Otherwise the singular link is called {\em admissible}. 
A crossing change on a link that produces an inadmissible singular link as intermediate step will be called  an
{\em inadmissible crossing change}.
\end{define}

In the proof of Theorem \ref{integration} it will be convenient for us to work
with framed links with ordered components: Let ${\tilde {\mathcal L}}$ denote the set of isotopy classes of such
framed links in $M$. Similarly, let ${\tilde {\mathcal L}}^{(n)}$ denote the set of isotopy classes of ordered 
framed singular links with $n$-double points.
There is an obvious map ${\mathfrak r}: {\tilde {\mathcal L}} \longrightarrow {\mathcal L}$ that forgets the 
ordering of the components
of  links; similarly we have forgetful maps  ${\mathfrak r}_n: {\tilde {\mathcal L}}^{(n)} \longrightarrow {\mathcal L}^{(n)}$, for all $n\in {\N}$.
Recall from the Introduction that the framing of a knot is determined by an integer, where in the case of not homologically trivial knots
this integer is provided by Chernov's work. Thus the framing
of an $m$-component link in ${\tilde {\mathcal L}}$ is determined by  an ordered
sequence  $\{{\bf f}_1, \ldots, {\bf f}_m\}$ of $m$ integers; one assigned to each component of the link. Every entry of the sequence is the affine self-linking number of a link component 
and it changes by $2$ under an inadmissible crossing change while it remains unchanged under admissible crossing changes (Theorems 2.2, \cite{chernov}).
Then, via ${\mathfrak r}$, an unordered link $L\in {\mathcal L}$ inherits an unordered sequence of integers: More specifically,
given $L\in {\mathcal L}$, there is a set of ordered integer sequences, say ${\bf f}$,  corresponding to elements in ${\mathfrak r}^{-1}(L)$.
We assign to $L$ the map 

$${\mathfrak r}^{-1}(L)\longrightarrow {\bf f},$$ 
sending each element to its corresponding ordered sequence.
We will often abuse the terminology and refer to ${\bf f}$
as the {\em framing} of the link $L$. 
\begin{define}\label{total} {\rm The  {\em total framing} of a link  $L\in {\mathcal L}$ is defined to be  $ {\tau}(L):= \sum_{i=1}^m {\bf f_i}$
where $\{{\bf f}_1, \ldots, {\bf f}_m\}$ is the ordered sequence corresponding to an appropriate lift
${\tilde {L}}\in {\mathfrak r}^{-1}(L)$ of $L$.}
\end{define}

\begin{define} \label{extend}For an ordered, framed singular link ${\tilde L}_{\times}\in {\tilde{\mathcal L}}^{(1)}$ 
we define a sequence of integers $\{{\bf f}_1, \ldots, {\bf f}_m\}$  by

$$ {\bf f}_i({\tilde L}_{\times}):= \; 
\left\{ 
\begin{array}{cl}{\bf f}_i({\tilde L_{+}})-{\bf f}_{i}({\tilde L}_{-}), 
  \mbox{ if} \quad  \times \in {\tilde L}_i \\
  \mbox  \quad  \quad \quad \quad \\
  {\bf f}_i({\tilde L_{+}}) ={\bf f}_i({\tilde L_{-}}), \mbox   \quad    {\rm otherwise.}\\
\end{array}
\right.
$$
Note that, in the first case, $ {\bf f}_i({\tilde L}_{\times})$ is non-zero only if $L_{\times}$ is inadmissible, in which case it is equal to 2.
For an unordered singular link  $L_{\times}\in {{\mathcal L}}^{(1)}$ we have 
a set of ordered integer sequences, say  ${\bf f}$,  corresponding to elements in ${\mathfrak r}^{-1}(L_{\times})$.
The map ${\mathfrak r}^{-1}(L_{\times})\longrightarrow {\bf f}$, assigning to every ordered link in that preimage its corresponding sequence, gives an unordered sequence of integers for $L_{\times}$.
\end{define}

\subsection{Integration of singular link invariants.}
Given an abelian group  $\mathbb A$
and a framed link invariant
$F: {\mathcal L} \longrightarrow \mathbb A$, we
can extend it to an invariant of framed singular links by defining
$$f(L_{\times}):= F(L_+)-F(L_-), \eqno(1)  $$
for every  $L_{\times}\in {\mathcal L}^{(1)}$. Continuing inductively we can extend the invariant on singular links in
${\mathcal L}^{(n)}$ for all $n\in {\bf N}$. 
We are interested in reversing this process; the reverse process is usually referred to as integration of the singular link invariant to an invariant of links
\cite{bn, ka, kalin, lin}.
In this section we deal with the following question: Suppose that we are given an invariant of framed singular links
$f: {\mathcal L}^{(1)} \longrightarrow \mathbb A$.
Under what conditions is there
a framed link invariant $F:{\mathcal L} \longrightarrow \mathbb A$ so that
(1) holds for all singular links $L_{\times} \in {\mathcal L}^{(1)}$?  We will address this question for links in 
$\Q$-homology $3$-spheres
with trivial $\pi_2$.

\begin{define}Let $N$ be an oriented compact  3-manifold with or without boundary.
A map $ \Phi:S^1\times S^1 \longrightarrow N$ is called
{\em essential} if it induces an injection on $\pi_1$ and it cannot be homotoped to a map 
$\Phi': S^1\times S^1 \longrightarrow \partial N$. Otherwise $\Phi$ is called {\em inessential}.
The manifold $N$ is called {\em atoroidal} if there are no {\em essential}
maps  $S^1\times S^1 \longrightarrow N$.
\end{define}

\begin{remark} \label{relation2}{\rm  
Let $L_{\times \times}\in{\mathcal L}^{(2)}$ be a framed singular link with two inadmissible singular points. By resolving the singular points, one at a time,
we obtain four singular links in ${\mathcal L}^{(1)}$.
These are shown in Figure 3, where the notation is consistent with that of Figure 2.
\begin{figure}[htbp]
\centering
\includegraphics[width=2.0in, ]{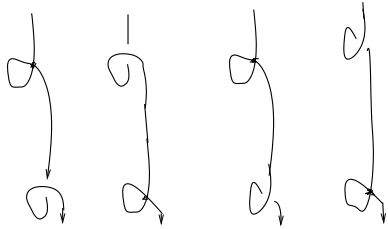} 
 \caption{From left to right: $L_{\times r}$, $L_{r\times }$, $L_{\times l}$,
$ L_{l \times }$}
\end{figure}
We note that $L_{{\times}r}$ is equivalent to $L_{{r\times}}$.
Similarly, $L_{{\times}l}$ is equivalent to $L_{{l\times}}$. Thus if
$f: {\mathcal L}^{(1)} \longrightarrow {\mathbb A}$ is an invariant of framed singular links
then we have
$$f(L_{\times r})=f(L_{r \times})\ {\rm and}\  f(L_{\times l})=f(L_{l \times}). \eqno(2)$$
Now (2) implies that the signed sum of $f$ on the four singular links in Figure 3, where signs are determined by (1),  is equal to zero.
Next we will show that if this  holds true for all $L_{\times \times}\in  {\mathcal L}^{(2)}$,
then $f$ can be integrated to a framed link invariant.}
\end{remark}

\begin{theorem}\label{integration} Suppose that  $M$
is a ${\mathbb Q}$-homology sphere
with $\pi_2(M)=0$ and such that that if $H_1(M)\neq 0$ then
$M$ is  atoroidal.
Let $f: {\mathcal L}^{(1)} \longrightarrow {\mathbb A}$
be an invariant of framed singular links with one double point. Suppose that 
${\mathbb A}$ is torsion free and that the invariant $f$ satisfies the relation
$$f(L_{{\times}+})-f(L_{{\times}-})=f(L_{+{\times}})-f(L_{-{\times}}), \eqno (3)$$
for every 
$L_{\times \times}\in {\mathcal L}^{(2)}$. Then there exists a framed link invariant $F$ such that $f$ is derived from $F$ via equation (1).
Here, the four singular links appearing in (3) are obtained by resolving the singular points of
$L_{\times \times}$ one at a time.
\end{theorem}

Theorem \ref{integration} is the framed link analogue of Theorem 3.16 of \cite{ka} and Theorem 3.1.2 of \cite{kalin}. As explained in the Introduction, however,
here we work in a more general class of manifolds. Also the presence of framing requires an adaptation of the arguments:
to formulate  the correct ``global integrability condition"  (equation (6) below)  we need
a notion of global framing around
homotopies of links. The definition of such a notion  is facilitated by the works of Chernov and Kaiser \cite{chernov, kaiser} (Definition \ref{framedhomo}).
For arguments  that are very similar to these in \cite{ka, kalin} we will refer the reader to
these articles for details. 

\subsection{Loop space and framing control}
Because in this section we work with oriented links  we need to slightly modify the set of initial links   $ \mathcal C \mathcal L^* \cup \{U\}$ chosen in the Introduction.
Recall that ${\mathcal L}$ (resp. ${\bar {\mathcal L}}$) denotes the set of isotopy classes of framed  oriented (resp. unoriented) links in $M$.
Consider the set of oriented links ${\mathcal C}{\mathcal L}:= {\mathfrak o}^{-1}(\mathcal C \mathcal L^* \cup \{U\})$, where ${\mathfrak o}:  {\mathcal L} \longrightarrow {\bar {\mathcal L}}$
is the obvious forgetful map. Also recall that
${\tilde {\mathcal L}}$ denotes the set of isotopy classes of ordered
framed links in $M$ and that we defined a forgetful map ${\mathfrak r}: {\tilde {\mathcal L}} \longrightarrow {\mathcal L}$.
Given $CL\in {\mathcal C}{\mathcal L}$, we pick $L\in {\mathfrak r}^{-1}({\mathcal C}{\mathcal L})$.
We will also use $L$ to denote a representative $L:P \longrightarrow
 M$ of $L$, where $P$ is a disjoint union of oriented circles.  Let ${{\mathcal M}}^L(P, M)$ denote the space of  ordered smooth framed immersions 
$P\longrightarrow M$ homotopic to $L$, equipped with the compact-open
topology. 
For every $L'\in {\tilde {\mathcal L}}$ and representative $L'\in {\mathcal M}^L(P, M)$,
let  $\Phi:  P {\times}[0, 1]
\longrightarrow M$ be a homotopy with  $\Phi(P\times \{0\}) =L'$ and $\Phi(P\times \{1\} )=L$.
After a small perturbation we can assume that for only finitely many points
$0<t_1<t_2<\cdots<t_n<1$, ${\phi}_{t}:= \Phi(P\times \{t\}) $ is not an embedding and it is a singular framed
link of order $1$. For different $t's$ in an interval of $[0, \ 1]\setminus \{t_1, \ t_2, \
\dots, t_n\}$ the corresponding framed links are equivalent and when  $t$ passes
through $t_i$, ${\phi}_t$ changes from one resolution of ${\phi}_{t_i}$
to the other.  

For $CL\in {\mathcal C}{\mathcal L}$,  
let ${\mathcal M}^{CL}(M)$ denote the space of unordered smooth framed immersions homotopic to
$CL$, equipped with the compact-open
topology. The projection $ {\mathfrak q}: {\mathcal M}^L(P, M)\longrightarrow {\mathcal M}^{CL}(M)$ is a covering map away from points that are fixed
under permutation of components.

\begin{define} \label{framedhomo}Let 
$\Phi$ be a homotopy between ordered links $L_1, L_2\in {{\mathcal M}}^L(P, M)$ 
with points
$0<t_{1}<\cdots<t_{n}<1$
such that  ${\phi}_{t_{j}} \in {\tilde  {\mathcal L}} ^{(1)} $.  For each singular link ${\phi}_{t_{j}}$ we have a sequence
$\{{\bf f}^j_i| i=1, \ldots, m\}$ as in Definition \ref{extend}.
We define the {\emph total framing} of $\Phi$  to be the sequence of integers 
$\{\Delta{\bf f}_i|i=1, \dots, m\}$, where
$$\Delta {\bf f}_i:=
\sum_{j=1}^{n}\delta_j^i{\epsilon}_{j} {\bf f}^{j}_i({\phi}_{t_{j}}). \eqno (4)$$
Here $\delta_j^i=1$ if the $i$-th component of ${\phi}_{t_{j}}$
contains the double point  and  0 otherwise. Also
${\epsilon}_{j}=1$ if
${\phi}_{t_{j}+\delta}$, for $\delta >0$ sufficiently small, is a positive
resolution of ${\phi}_{t_{j}}$ and
${\epsilon}_{j}={-1}$ otherwise. We will say that the total framing is zero
iff $\Delta {\bf f}_i=0$, for all $1, \dots, m$.

Given a loop $\Phi\in {\mathcal M}^{CL}(M)$ we obtain a set
of ordered sequences $\Delta{\bf f}_{\Phi}$ associated to the set of all lifts of $\Phi$ in $ {\mathcal M}^L(P, M)$.
The map ${\mathfrak q}^{-1}(\Phi)\longrightarrow \Delta{\bf f}_{\Phi}$ defines an unordered
sequence of integers for $\Phi$.
The homotopy  $\Phi$ is called {\em framing preserving} 
iff the total framing of every element in ${\mathfrak q}^{-1}(\Phi)$ is zero.
We will write
$\Delta{\bf f}_{\Phi}={\bf 0}$.
\end{define}

\subsection {Beginning the proof of Theorem \ref{integration} }  We want to define an invariant $F:{\mathcal L} \longrightarrow {\mathbb A}$
that is obtained from the given $f: {\mathcal L}^{(1)} \longrightarrow {\mathbb A}$ via (1).
First we assign values of  $F$ on the set of  initial links ${\mathcal C}{\mathcal L}$.
Now fix $CL\in {\mathcal C}{\mathcal L}$ and let
$L' \in {\mathcal M}^{CL}(M)$ be a framed link.
Choose a generic  homotopy $\Phi$
from $L'$ to $CL$.
Let $0<t_1<t_2<\cdots<t_n<1$ denote the points where ${\phi}_t$ is not an embedding. Recall that
${\phi}_{t_{i}} \in {\mathcal L} ^{(1)} $ such that
for different $t's$ in an interval of $[0, \ 1]\setminus \{t_1, \ t_2, \
\dots, t_n\}$, the corresponding framed links are equivalent. When $t$ passes
through $t_i$, ${\phi}_t$ changes from one resolution of ${\phi}_{t_i}$
to another.  
We define
$${F(L')=F(CL)+\sum_{i=1}^{n}{\epsilon}_i f({\phi}_{t_i})}\eqno (5)$$
Here ${\epsilon}_i={\pm}1$ is determined as follows: If
${\phi}_{t_i+\delta}$, for $\delta >0$ sufficiently small, is a positive
resolution of ${\phi}_{t_i}$ then ${\epsilon}_i=1$. Otherwise
${\epsilon}_i={-1}$.

To prove that $F$ is well defined we have to show that modulo ``the
integration constant" $F(CL)$, the definition of $F(L')$
is independent of the choice of the homotopy. For this we consider a
closed homotopy $\Psi$ from $CL$ to itself. After a small perturbation,
we can assume that there are only finitely many points
$x_1,x_2,\dots,x_n \in S^1$, ordered cyclicly according to the
orientation of $S^1$, so that ${\psi}_{x_i}\in {\mathcal L}^{1}$ and
$\psi_x$ is equivalent to $\psi_y$ for all $x_i<x,y<x_{i+1}$. To prove that $F$
is well defined
we need to show that
$${X_{\Psi}:=
\sum_{i=1}^{n}{\epsilon}_i f({\psi}_{t_i})=0} \eqno(6)$$
where ${\epsilon}_i={\pm}1$ is determined by the same rule as above.
\vskip 0.04in

{\underline{Independence of link component orderings:}}  To prove (6) we will turn our attention to ordered links:
First we note that the invariant $f$ pulls back to an invariant on ${\tilde {\mathcal L}}^{(1)}$ via the forgetful map ${\mathfrak  r}$.
After iterating $\Phi$ several times if necessary  we can assume that it  lifts to a loop in
$ {\mathcal M}^L(P, M)$ based at $L$ (compare,
page 3874 of \cite{kaiser}). Given a self-homotopy $\Phi$ of $CL$ and the associated quantity $X_{\Phi}$,  lift
$\Phi$ to a closed homotopy $\Psi$ in  ${\mathcal M}^L(P, M)$ and let $X_{\Psi}$ denote the lift of $X_{\Phi}$.
Note that 
$X_{\Psi}=a X_{\Phi}$, for some integer $a\in {\Z}$. Since ${\mathbb A}$ is torsion free we have
$X_{\Phi}=0$ exactly when $X_{\Psi}=0$. Thus, it is enough to check (6) for  
 homotopies that preserve the ordering of components.
 \vskip 0.04in
 
 {\underline{Restriction to framing preserving homotopies:}} Next we observe that it is enough to check (6) for homotopies that are framing preserving in the sense of Definition \ref{framedhomo}: To see that we recall that given a framed link 
 $L'\in {\mathcal M}^{CL}(M)$ we need to check that  (5) does not depend on  the homotopy from $L'$ to the framed link $CL$ used to define it.
Thus the closed homotopies $\Phi$  that we need (6) to hold for, are those obtained by composing two homotopies from $L'$ to $CL$. Each component of $CL$ is equipped with
a vector field and going around $\Phi$ does  not change the homotopy class of this vector field (that is the equivalence class of $CL$ as a framed link). We can think that the framing of $CL$ transports
to a ``new" framing around $\Phi$. The two framings might differ by twists on the components of $CL$ but the total singed number of the twists must be zero.
The total sum of such twists  is captured exactly by the quantity $\Delta{\bf f}_{\Phi}$ (compare, Theorem 6 of \cite{kaiser}).
The framing of $CL$ lifts to one on $L$ and
 going around the self-homotopy of $L$ that lifts $\Phi$ also preserves the homotopy class of the framing vector field.  

 
 The proof of (6), which occupies the remaining of Section 2 and Section 3, will be divided into several steps. In this section we will  give the proof of (6) for closed homotopies
 in atoroidal 3-manifolds and in the next section we deal with essential tori. 

To continue,  suppose that $P$ has $m$
components $P =\sqcup_{i=1}^{m} P_i$,
where each $P_i$ is an oriented circle. Let $L: P
\longrightarrow M$ be a link.
Pick a base point $p_i \in P_i$ and
let $a_i$ denote the homotopy class of $L(P_i)$ in
$\pi_1 (M, L(p_i))$. We denote by $Z(a_i)$ the centralizer of
$a_i$ in $\pi_1(M, L(p_i))$.
We begin with the following lemma (see, for example, the proof of Proposition 4.3 of \cite{lin}).

\begin{lemma}  \label{centralizerI}Suppose that $M$ is an
orientable 3-manifold
with $\pi_2(M)=0$ and let the notation be as above.
Then
$$\pi_1({\mathcal  M}^L(P, M), L) \cong \oplus_{i=1}^{m}Z(a_i).$$
\end{lemma}

\subsection{Integrating around inessential tori}
Here we show how to derive (6) in the case where the closed homotopy
$\Phi$ represents a collection of inessential tori  in $M$. Since $\partial M=\emptyset$ this means
that the induced map $({\Phi_i})_*: \pi_1(P_i \times S^1) \longrightarrow  \pi_1(M)$
has non-trivial kernel. Here $\Phi_i:= \Phi | P_i \times S^1$, for $i=1, \dots, m$.

\begin{lemma} \label{disc} Let $\Phi$ be a loop in ${\mathcal  M}^L(P, M)$ representing a framing preserving self-homotopy  of  $L$. Suppose that $\Phi$ can be extended to a map ${\hat \Phi}: P\times D^2
\longrightarrow M$ where $D^2$ is a 2-disc with $\partial D^2 = \{*\} \times
S^1$. Then $X_{\Phi}=0$.
\end{lemma}
\begin{proof}
We perturb ${\hat \Phi}$, relatively $\partial D^2$,  so that it is in general position
in the sense of Proposition 1.1 of \cite{ka}. Then the set
$$S_{{\hat \Phi}}:= \{ x \in D^2 \  | \ {\hat \phi}_x:= {\hat \Phi}(P\times \{x\})  \  {\rm is \  not\  an \  embedding}\},$$
is a graph in $D^2$ with properties (1)-(5) given in Proposition 1.1 of \cite{ka}.
The vertices of $S_{{\hat \Phi}}$ in the interior of $D^2$ are of valence one or  four (see Figure 4).

\begin{figure}[htbp] %
\centering
\includegraphics[width=2.00in, ]{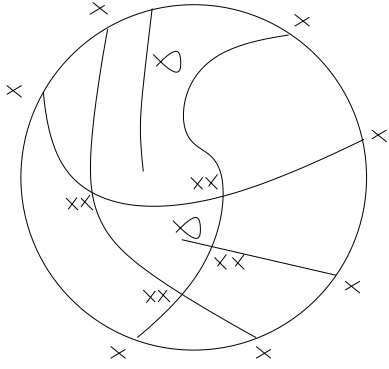} 
\caption{The set of singularities $S_{{\hat \Phi}}$ with the types of double points they represent. }
\end{figure}
The invariant
$f$ assigns an element of $\mathbb A$ to every edge of $S_{{{\hat \Phi}}}$.
We observe that  condition (3) in the statement of Theorem \ref{integration}
implies that 
$X_{{\Phi}}$ is independent on the order in which the crossing changes
around $\Phi:={\hat \Phi}| P\times \partial D^2$ occur. Thus, without loss of generality, we may assume that
the valence one vertices of $S_{{{\hat \Phi}}}$ in the interior of $D^2$
correspond to
inadmissible crossing changes on 
$\partial D^2$. 
With the notation as above,  we will assume that the framed singular link  ${\phi}_{x_i}\in {\mathcal L}^{1}$
is inadmissible
for $i=1, \ldots, s$  and admissible  for 
$i=s, \ldots, n$. In particular,  there are  $s$ edges of $S_{{{\hat \Phi}}}$ emanating from
$x_1,\dots,x_s$ respectively  and ending at an interior vertex of valence one,  and these are the only valence one vertices of $S_{{{\hat \Phi}}}$.
\begin{figure}[htbp] %
\centering
\includegraphics[width=2.3in, ]{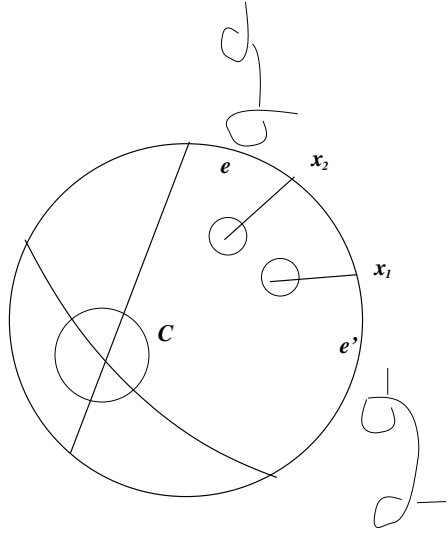} 
\caption{The singular links ${\phi}_{x_1}$, ${\phi}_{x_2}$
form a pair of type
 $L_{\times r}, L_{r \times }$
(or $L_{\times l},  L_{l \times }$) shown in Figure 3. The framed links corresponding to the components  $e, e'$ of  ${\partial D^2} \setminus \{x_1, x_2, \dots \}$
are isotopic.}
\end{figure}
   
For every interior
vertex of $S_{{{{\hat \Phi}}}}$ we draw a small circle $C$ around it so that the
number of
points in $C\cap S_{{ {{\hat \Phi}}}}$ is equal to the valence of the vertex. See Figure 5.
Let $C_1,\ldots, C_s$ denote the circles surrounding the valence one vertices
of $ S_{{{{\hat \Phi}}}}$ and let ${\Gamma}$  denote the disjoint union of the circles surrounding
the vertices of valence four. For a vertex of valence four the four points in
$C\cap S_{{\hat \Phi}}$ correspond  exactly to these appearing in equation
(3). Thus by (3) we have
$${\sum _ {x\in \Gamma \cap S_{{{{\hat \Phi}}}}}{\epsilon_x}f({{\hat \phi}}_{x})=0,} \eqno(7)
$$
where
$ {\hat \phi}_x:={{ {\hat \Phi}} (P\times \{x\})}.$
Now observe that 
$$ \sum_{i=s+1}^{n}{\epsilon}_i f({{\phi}}_{x_i})=\sum _ {x\in \Gamma \cap S_{{{{\hat \Phi}}}}}{\epsilon_x}f({{\hat \phi}}_{x})=0.$$
The last equation and (7) imply that
 $$X_{\Phi}= \sum_{i=1}^{s}{\epsilon}_i f({{\phi}}_{x_i}). \eqno(8)$$
Since   $\Phi$  is framing preserving we have  $\Delta{\bf f}_{\Phi}={\bf 0}$.
By Definitions \ref{extend} and  \ref{framedhomo} and the fact that ${\bf f}$ remains unchanged under admissible crossing 
changes we have
$\Delta{\bf f}_{C}={\bf 0}$, for every loop
$C\in \Gamma$.
This in turn  implies that
$$\Delta{\bf f}_{\Gamma}:= \sum _ {C\in \Gamma}\Delta{\bf f}_{C} ={\bf 0}$$
Since we have
$$\Delta{\bf f}_{\Phi}=\sum_{i=1}^{s}{\epsilon}_i {\bf f}({{\phi}}_{x_i})+\Delta{\bf f}_{\Gamma}={\bf 0}$$
we  conclude that $ \sum_{i=1}^{s}{\epsilon}_i {\bf f}({{\phi}}_{x_i})={\bf 0}$. This in turn implies that the inadmissible singular links 
${{\phi}}_{x_i}$ can be partitioned into pairs of the forms shown in Figure 3.
Relation (2) in Remark \ref{relation2} shows that the right hand side of (8) is identically zero.  Thus $X_{\Phi}=0$,  as desired.
\end{proof}

\begin{remark} \label{nopreserving} {\rm Let ${\bar X}_{\Phi}$ denote the contribution of the admissible singular links around $\Phi$ to $X_{\Phi}$.
The proof of Lemma \ref{disc} shows that regardless of whether $\Phi$ is framing preserving, relation (3) implies that ${\bar X}_{\Phi}=0$.}
\end{remark}

\begin{remark} \label{smooth}Proposition 1.1 of \cite{ka}, referenced in the proof of Lemma \ref{disc}, is stated in there for  the PL-category.
However, as explained by Kaiser in Section 3 of  \cite{kaiser1}, the statement is true in the smooth category  which is actually what we need here.
We should also remark that, as explained by Lin in \cite{lin},  the conclusion holds if the disc $D^2$ is replaced by any planar surface $F$.
Furthermore, if $\Phi|\partial F$ is already in general position then the modifications that put $\Phi$ into general position on $F$ can be performed relatively $\partial F$.
\end{remark}

A  slight variation of the proof   of Lemma \ref{disc} shows the following:

\begin{lemma} \label{disc1} Let $\Phi$ be a loop in ${\mathcal  M}^L(P, M)$ representing a framing preserving self-homotopy  of a framed link $L$. 
Let $P':=P\setminus P_1$. Suppose that $\Phi|P'$ can be extended to a map ${\hat \Phi}: P'\times D^2
\longrightarrow M$ where $D^2$ is a 2-disc with $\partial D^2 = \{*\} \times
S^1$. Suppose moreover that $\Phi| (P_1\times S^1)$ is an embedding. Then $X_{\Phi}=0$.
\end{lemma}

 The proof of the next lemma is given in the proof of  Lemma 3.3.4 of \cite{kalin}. 

\begin{lemma} \label{tori1} Let $M$ be a $\Q$-homology 3-sphere
with $\pi_2(M)=0$. Suppose that  $\pi_1(M)$ is infinite and that  $L$ has no homotopically trivial components. Let $\Phi \subset {\mathcal  M}^L(P, M)$ be a
framing preserving  closed homotopy  such that
the restriction
$\Phi| P_i\times S^1\longrightarrow M$  is inessential, for all $i=1,\dots, m$.
There exists a 2-disc $D^2$ and a map 
${\tilde {\Phi}}: P\times D^2 \longrightarrow M$ such that
$$X_{\partial\tilde{\Phi}}= a X_{{\Phi}}, \eqno(9)$$
for some $a\in\Z$. Here ${\partial\tilde{\Phi}}={\tilde {\Phi}}| P\times 
{\partial
D^2}$.
\end{lemma}

\subsection {Theorem \ref{integration} for atoroidal manifolds} 
Before we can proceed with the proof of the theorem we need two additional lemmas.

\begin{lemma} \label{homotopy} Consider  ${\Phi},{\Phi'}: S^1 \longrightarrow  {{\mathcal M}^L(P, M)}$
two self-homotopies of $L$.  Let ${\bar X}_{\Phi}$ and ${\bar X}_{\Phi'}$ denote
the contribution to $X_{\Phi}$ and ${ X}_{\Phi'}$
coming from admissible singular links around $\Phi$ and $\Phi'$,  respectively.
Suppose that ${\Phi}, {\Phi'}$ are freely homotopic as loops in  ${\mathcal  M}^L(P, M)$. Then
we have ${\bar X}_{\Phi'}={\bar X}_{\Phi}$.
Furthermore, there is a group homomorphism $\psi : \pi_1({\mathcal  M}^L(P, M), \ L) \longrightarrow {\mathbb A}$ defined by $\psi([\Phi]):={\bar X}_{\Phi}$.
\end{lemma}

\begin{proof} By a slight variation of the argument in the proof of  Lemma 3.3.2 of \cite {kalin} we have the following:
There  exists a map ${\hat {\Psi}}: D^2 \longrightarrow  {{\mathcal M}^L(P, M)}$
such that if we set $\Psi:= {\hat {\Psi}}|{\partial D^2}$ then 
$\Psi: S^1 \longrightarrow  {{\mathcal M}^L(P, M)}$
is a self-homotopy of $L$ 
with $$X_{\Psi}=X_{\Phi}-X_{\Phi'}.$$
Lemma \ref{disc} and Remark \ref{nopreserving}
imply  ${\bar X}_{\Psi}=0$; thus ${\bar X}_{\Phi}={\bar X}_{\Phi'}$.

For the remaining of the claim define  $\psi : \pi_1({\mathcal  M}^L(P, M), \ L) \longrightarrow {\mathbb A}$ as follows:
Given $\alpha\in \pi_1({\mathcal  M}^L(P, M), \ L)$, let $\Phi$ is be a self-homotopy of $L$
representing $\alpha$. Define $\psi(\alpha)={\bar X}_{\Phi}$. By our earlier arguments $\psi(\alpha)$ is independent on the representative $\Phi$.
The fact that $\psi$ is a group homomorphism follows easily.
\end{proof} 

The next  lemma  is Lemma 3.2.5 in \cite{kalin}. We point out that the proof of this lemma 
uses the hypothesis that the group $\mathbb A$ in which the invariants take values is torsion free.

\begin{lemma} \label{centralizer} 
Suppose that $M$ is a $\Q$-homology
3-sphere with $\pi_2(M)=0$. Let $L: P \longrightarrow M$ be a framed link and let
$\Phi: P{\times}S^1\longrightarrow M$
be a framing preserving self-homotopy of $L$. Assume that,
for some $i=1, \ldots,  m$, we have $a_i=1$.
Set $P':= P\setminus P_i$ and $\Phi':= \Phi| P'$. If $X_{\Phi'}=0$ then $ X_{\Phi}=0$.
\end{lemma}

We are now ready to give the proof of Theorem \ref{integration}
in the case where $M$ is an atoroidal $\Q$-homology 3-sphere. 

\begin{theorem} \label{specialcase} Suppose that  $M$ is an atoroidal  $\Q$-homology
3-sphere with $\pi_2(M)=0$. Then the conclusion of Theorem \ref{integration}
is true for $M$.
\end{theorem}
\begin{proof}
Let $f: {\mathcal L}^{1} \longrightarrow {\mathbb A}$ be a framed singular link 
invariant satisfying (3) of the statement of Theorem \ref{integration} and let  $\Phi : P\times S^1 \longrightarrow M$
be a  framing preserving self-homotopy of a framed link $L: P \longrightarrow M$.
We have to show that
 $$X_{\Phi}=0$$
where $X_{\Phi}$ is the signed sum of
values of $f$ around $\Phi$ defined
in (6). 

First suppose that $\pi_1(M)$ is finite. Then, by Lemma \ref{centralizerI},
$\pi_1({\mathcal  M}^L(P, M), \ L)$  is finite.  Since  ${\mathbb A}$
is torsion free the  homomorphism $\psi$ 
of  Lemma \ref{homotopy} must be the trivial one. Thus, in particular,
$ X_{\Phi}=0$.

Now  suppose that $\pi_1(M)$ is infinite.
If the link $L$ to begin with contains no homotopically trivial components, then since $M$ is atoroidal, 
Lemma \ref{tori1} applies to conclude that 
$X_{\partial\tilde{\Phi}}= a X_{{\Phi}}$, for a map ${\tilde {\Phi}}: P\times D^2 \longrightarrow M$.
By Lemma \ref{disc}, $X_{\partial\tilde{\Phi}}= a X_{{\Phi}}=0$
and thus, since ${\mathbb A}$ is torsion free, $ X_{{\Phi}}=0$.

Next suppose that all the components of $L$ are homotopically trivial; that is $a_i=1$, for $i=1, \ldots, m$. Then, by Lemma \ref{centralizerI},
$$\pi_1({\mathcal  M}^L(P, M), L) \cong \oplus_i^{m}\pi_1(M, L(p_i)).$$
Since $H_1(M)$ is finite the above equality implies that the abelianization of the group
$\pi_1({\mathcal  M}^L(P, M), L)$ is a finite group.
By Lemma \ref{homotopy} we have a homomorphism
$\psi : \pi_1({\mathcal  M}^L(P, M),  L) \longrightarrow {\mathbb A}$ with
$\psi([\Phi ])=X_{\Phi}$. Since ${\mathbb A}$ is abelian
$\psi$ factors through the abelianization of $ \pi_1({\mathcal  M}^L(P, M), L)$; a
finite group. 
But since ${\mathbb A}$ is torsion free $\psi$ is the trivial homomorphism. Thus
$X_{\Phi}=0$.

To handle the general case let $h(L)$ denote the number of components of $L$ that are homotopically trivial.
The proof is by induction on $h(L)$. In the light of our discussion above, the conclusion is true if $h(L)= 0$ or $h(L)=m$.
Thus we may assume that $h(L)\neq 0, m$. 
Let $L_i\subset L$ be a component that is homotopically trivial and let $L': =L\setminus L_i$.
Also let $\Phi$ be a self-homotopy of $L$ and let $\Phi'$ denote the restriction of $\Phi$ on $P'$,
where $P':= P\setminus P_i$. Since $h(L')<h(L)$, by induction, 
 $X_{\Phi'}=0$. Then, by  Lemma  \ref{centralizer},  $X_{\Phi}=0$.
\end{proof}

\section{Integration of invariants in toroidal 3-manifolds}
To study the question of integrability of singular link invariants
in  toroidal 3-manifolds we need 
several results from the theory of  the characteristic submanifold of Jaco-Shalen \cite{jaco-shalen} and Johannson \cite{joh}. The statements 
of the results from these theories, in the form needed in our setting, are summarized in Section 2 of \cite{ka} and in Section 2 of \cite{kalin}. It will be convenient for us to recall the statements we need below
from therein, instead from the original references. In particular we will need the
Enclosing Theorem and the Torus Theorem both stated on pp. 679 of \cite{ka}. The later, in the form needed for our purposes,  follows from work of Scott, Casson-Jungreis and Gabai.

\begin{theorem}\label{integration2} Let  $M$ be a ${\mathbb Z}$-homology 3-sphere
with $\pi_2(M)=0$ and let ${\mathbb A}$ be a torsion free abelian group. Suppose that a map $f: {\mathcal L}^{(1)} \longrightarrow {\mathbb A}$ satisfies (3) of Theorem \ref{integration}. Then there exists a framed link invariant $F$ such that $f$ is derived from $F$ via equation (1).
\end{theorem}

\begin {remark}\label{necessary}  {\rm The restriction to ${\mathbb Z}$-homology 3-spheres in Theorem \ref {integration2} is necessary.
As explained in Remark 3.13 of \cite{ka} and the discussion at the end of Section 3 in \cite{kalin}, in general, local conditions are not sufficient for the integration of singular link invariants.
When the characteristic submanifold contains Seifert fibered components over non-orientable surfaces one needs to impose extra non-local conditions. Specific constructions demonstrating these phenomena are given by Kirk and Livingston in \cite{kirkliv}. The necessity of working with irreducible 3-manifolds is demonstrated by \cite{kirkliv} as well as the work of Eiserman \cite{ei}.}
\end{remark}

The proof of Theorem \ref{integration} will be completed once we have proved Theorem \ref{integration2}. For the proof of Theorem \ref{integration2} we will need the following:

\begin{lemma} \label{essential}
Let $M$ be a $\mathbb Z$-homology 3-sphere with $\pi_2(M)=0$. Suppose that  $\Phi : T=S^1\times S^1 \longrightarrow M$ is an essential map. Then there exists a map $\Psi : T \longrightarrow M$
homotopic to $\Phi$ and such that one of the following holds:
\begin{enumerate}
\item $\Psi(T)$ lies on an essential embedded torus in $M$.
\item There exists an oriented surface $F$  with $\partial F \neq \emptyset$,  and a trivial fiber bundle $Y=S^1 \times F$,
with the following property: $\Psi$ extends  to a map ${\hat {\Psi}}: Y
\longrightarrow M$ so that  the image ${\hat {\Psi}(\partial Y \setminus T})$ is contained on a collection of embedded tori in $M$.
\end{enumerate}
\end{lemma}

\begin{proof}  By the Torus Theorem and the discussion at the end of Section 2 of \cite{kalin}, either $M$ is Haken or it is a Seifert fibered 3-manifold that fibers over $S^2$ with
three or less exceptional fibers. 

First suppose that $M$ is Haken. Then by the Enclosing Theorem
there is a Seifert fibered submanifold $S\subset M$ and a homotopy
$\Phi'_t: T\longrightarrow M$ such that $\Phi'_0=\Phi$ and $\Phi'_1(T)\subset S$.
If $\Phi'_1(T)$ can be further homotoped in $S$ so that it lies on a component of $\partial S$ then we have conclusion (1). Otherwise,  
by the classification of essential tori in Haken Seifert fibered spaces
(Proposition 2.11 of \cite{ka}) we can homotope $\Phi'_1$ in $S$ to a map $\Psi: T\longrightarrow S$ which is vertical with respect to the fibration.

Next suppose that $M$ is a Seifert fibered space.
By Proposition 2.2.5 of \cite{kalin},  $\Phi $ is homotopic to a map
$\Psi: T\longrightarrow M$ which is vertical with
respect to the fibration of $M$. 

Thus, in both cases, either (1) holds or  we have a Seifert fibered manifold $S\subseteq M$, with orbit space $B$ and fiber projection $p$,  such that
$\Phi$ is homotopic  to a map  $\Psi : T \longrightarrow M$
that is vertical with respect to the fibration of $S$. 
This means that $\Psi$ is a composition $\Phi_1\circ q$,
where $q$ is a covering map from the torus $T$ to itself and $\Phi_1: T \longrightarrow S$ is an immersion without triple points.
Then, there exists a decomposition
$T=S^1\times S^1$ such that

a) $\Phi_1(S^1\times \{*\})$ maps onto a regular fiber $h$ of $S$;

b) we have $p({\Phi_1}(\{*\}\times S^1))= p(\Phi_1(T))$ on the orbit surface $B$ of $S$.

Let  $H$ (resp. $Q$) denote the curve 
$S^1\times \{*\}$ (resp. $\{*\}\times S^1$) on $T$.
Now $\alpha:=p(\Phi_1(T))$ is an immersed closed curve on $B$ with singularities finitely many transverse double points.
A neighborhood  $N:=N(\alpha) \subset B$ of $\alpha$ on $B$ is an oriented  planar surface.
Choose $N$ small enough so that $Y:=p^{-1}(N)$ 
contains no exceptional fibers of $S$. Now
$p: Y\longrightarrow N$ is an $S^1$-bundle and since $H^2(N)=0$
this bundle is trivial.  Choose a trivialization $Y\cong S^1 \times N$ so that $N$ is
embedded as a cross-section.
Pick a base point $b\in N$ and arcs from $b$ to the components
of $\partial N$ whose homotopy classes freely generate $\pi_1(N)$; we pick one arc for each such component.
Assume that these arcs intersect $\alpha$ only at its double points; let $x_1, \ldots, x_s$ denote the resulting generators of $\pi_1(N, b)$. Write $\alpha$ as a word in these generators; say
$$[\alpha]= x_{i_1}^{k_1}x_{i_2}^{k_2}\cdots x_{i_r}^{k_r}.$$
We can extend the restriction $\Psi| \{*\}\times S^1$ to  a map
 ${\hat {\Psi}}: (F, \ \partial F) \longrightarrow
( N, \ \partial N)$, 
where $F$ is a planar surface, such that:
(i)
the induced map ${\hat {\Psi}}_{*}: \pi_1(F)\longrightarrow \pi_1(N)$ is
onto; 
(ii) $\pi_1(F)$ is freely generated by elements
$a^1_1,\cdots,  a^1_{k_1}, a^2_1, \cdots , a^2_{k_2} , a^r_1,\cdots,  a^r_{k_r}$; 
(iii) ${\hat {\Psi}}_{*}([Q])=x_{i_1}^{k_1}x_{i_2}^{k_2}\cdots x_{i_r}^{k_r}$ (see proof of Lemma 3.11 of \cite{ka}). We pull back the fiber bundle  structure by  ${\hat {\Psi}}$ to obtain a fiber bundle ${\hat {\Psi}}^{*}(Y)\longrightarrow F$, over $F$. The pull-back of the cross-section $\alpha$ is a cross-section of 
${\hat {\Psi}}^{*}(Y)$. Extending this cross-section over $F$, and
conclusion. 
\end{proof}

We now recall that the proof of Theorem \ref{integration2} is reduced to showing (6) for
every framing preserving  self-homotopy of $L$. Using Lemmas \ref{centralizerI}, \ref{homotopy}, and \ref{centralizer} we will see that  the general case is essentially reduced
to the case of knots. Before we continue with the proof Theorem \ref{integration2} some remarks are in order.

\begin{remark} \label{canframe}{\rm  
Let $\Phi: P\times S^1 \longrightarrow M$ denote a framing preserving  self-homotopy of a framed link $L$
and let $\Phi'$ be obtained by a free homotopy of $\Phi$ in $M$. Consider the
homotopy from $\Phi$ to $\Phi'$ as a map ${\mathcal H}: P\times S^1 \times [0,1]\longrightarrow M$.  We can smoothly approximate ${\mathcal H}$
by a homotopy in general position as in the proof of Lemma \ref{disc} (see Remark \ref{smooth}).
Then we can view $\mathcal H$ as a family of smooth framed immersions $S^1 \longrightarrow  M$
parametrized by an annulus.  We note that the  closed homotopy $\Phi'$ is not necessarily framing preserving.}
\end{remark}

\begin{remark}\label{surface} {\rm   Suppose that we have a map $\Phi: Y:=S^1 \times F\longrightarrow M$,
such that $F$ is a planar surface so that there is a component $\alpha\subset \partial F$ such that
the restriction $\Phi | S^1\times \alpha$ is a loop in $M^L(P, M)$. We can view $\Phi$ as a family of framed immersions in $M$, parametrized by $F$.
We can cut $Y:=S^1 \times F\longrightarrow M$ along a collection of properly embedded annuli (the projection of which on $F$ decomposes $F$ into a disc) into a product $S^1\times D^2$. By considering the pull back of $\Phi$ on $S^1\times D^2$ we obtain a family of framed immersions
in $M$ parametrized by $D^2$. }
\end{remark}
In the next  lemma we treat homotopies that involve essential tori. The proof treats separately  the case of knots and that of links.
In the case of knots ($m=1$ below) the proof is very similar to that of Case 1 of 
Lemma 3.3.3 in \cite{kalin}. The starting ingredient in the proof of \cite{kalin} is Lemma 3.3.2 therein. Here we replace 
that ingredient with  Lemma \ref{essential} and we outline the argument below.
 
\begin{lemma} \label{tori} Let $M$ be a  $\mathbb Z$-homology 3-sphere
with $\pi_2(M)=0$ and let $\Phi: P\times S^1 \longrightarrow M$ be a framing preserving self-homotopy of a  framed link $L$. Suppose that
$\Phi_i:= \Phi| P_i\times S^1$ is an essential map, for some $i=1, \dots, m$.  Suppose, moreover, that $\Phi_i$ cannot be homotoped so that its image
lies on an essential embedded torus in $M$.
Then we have 
$X_\Phi=0$.
\end{lemma}
\begin{proof} Let $m$ be the number of components of $L$. We distinguish two cases according to whether  $m=1$ or $m>1$.

\vskip 0.04in

{\underline{ We have $m=1$:}} 
Since $\Phi$ is framing preserving, relation  (2) implies that
the total contribution of the inadmissible singular links along $\Phi$ to
 $X_{\Phi}$ is zero (proof of Lemma \ref{disc}). Thus, without loss of generality, we can assume that no inadmissible crossing changes occur along $\Phi$.
Now let $\Psi: P\times S^1\longrightarrow M$ be a map that is freely homotopic to $\Phi$ in $M$.  By Lemma \ref{homotopy}, and our earlier assumption on $\Phi$,  we have
${\bar X}_{\Psi}=X_{\Phi}$.

\vskip 0.04in

Set $T:=P\times S^1$, 
$l:=P\times \{*\}$ and $m:=\{*\}\times S^1$.  By assumption $\Phi| P\times S^1\longrightarrow M$ is an essential map
and it cannot be homotoped so that its image
lies on an essential embedded torus in $M$.
By Lemma \ref{essential} we can homotope $\Phi$ to a map
$\Psi: P\times S^1 \longrightarrow M$ so that:
There is a trivial fiber bundle $Y= S^1\times F$, over a planar surface $F$,
such that $\Psi$ extends to a map ${\hat {\Psi}}:  S^1\times F
\longrightarrow M$ and  the image ${\hat {\Psi}}{(\partial Y\setminus T})$ is contained on a collection
of embedded tori in $M$. 
Let $H$ denote a simple closed curve $T$ representing
a fiber of $Y$ and  let $Q$ denote the component of $\partial F$ (embedded as a cross-section of the bundle) on  $T$. In   
$\pi_1(T)$ we have $[l]=a [H] +b [Q]$, for some $a,b\in \Z$.

\vskip 0.02in

First suppose that $a=0$. Then  Lemma 3.12 of \cite{ka} (or Lemma 3.3.1 of \cite{kalin})
applies to conclude that ${\bar X}_{\Psi}=0$. By our discussion above, 
$X_{\Phi}={\bar X}_{\Psi}=0$  and the conclusion in this case follows.
\vskip 0.03in

Suppose now that $a\neq 0$.
Let $q : {\tilde Y} \longrightarrow Y$
be the covering of $Y$ corresponding to the subgroup
$a {\Z} \times  \pi_1(F)$ of $ \pi_1(Y)=\Z\times \pi_1(F)$. Lift
$l$, 
$H$, and $Q$ to curves
${\tilde {l}}$,
${\tilde {H}}$, ${\tilde Q}$, respectively,  on the torus ${\tilde T}:=q^{-1}(T)$.
Now $\tilde Y$ is a trivial fiber bundle over a surface ${\tilde F}$ with fiber $\tilde l$; we will
write $\tilde Y=\tilde l \times {\tilde F}$.  
Consider the composition ${\tilde \Psi}:={\hat \Psi} \circ q$ and its restriction on  $\tilde T\cong {\tilde l}\times {\tilde Q}$.
Since ${\tilde \Psi}({\tilde l}\times \{x\})=\Psi(l\times q(\{x\}))$, for all $x\in {\tilde Q}$,
the restriction ${\tilde \Psi}|{\tilde l}\times {\tilde Q}$ is a self-homotopy  of a framed knot; the parameter space is ${\tilde Q}$. 
As in Remark \ref{surface} we will
think of ${\tilde \Psi}$ as a family  of framed immersions parametrized by a disc $D^2$. Then we can consider $X_{\partial\tilde{\Psi}}$.
As in the proof of Lemma 3.14 of \cite{ka} we obtain that $X_{\partial\tilde{\Psi}}=c X_{\Psi}$, for some $c\in{\Z}$.
Since, as  discussed at the beginning of this proof we have ${{\bar  X}_{\Psi}}= X_{{\Phi}}$,
it follows that ${\bar X}_{\partial\tilde{\Psi}}=c X_{{\Phi}}$.
By Remark \ref{nopreserving}, we have ${\bar X}_{\partial\tilde{\Psi}}=0$.
Hence we conclude that
we have $c X_{{\Phi}}=0$ for some $c\in \Z$.
Since $\mathbb A$ is torsion free
this implies that $ X_{{\Phi}}=0$; finishing thereby the proof of the Lemma in the case $m=1$.
\vskip 0.05in

{\underline{ We have $m>1$.}} By Lemma \ref{centralizerI}, $\pi_1({\mathcal  M}^L(P, M), L)$ is isomorphic to a direct
product of the groups $\pi_1({\mathcal  M}^L(P_i, M), L_i)$ for  $i=1,\ldots, m$.  By Lemma \ref{homotopy}
it is enough to verify (6) only for homotopies $\Phi$ that are fixed on all but one component of $L$.
To that end, let $\Psi$ be a homotopy in general position that only moves one component; say $L_1$.
Suppose, without loss of generality, that
$\Psi| P_1\times S^1\longrightarrow M$ is an essential map that cannot be homotoped so that its image
lies on an essential embedded torus in $M$.
By (3), we may decompose $\Psi$ into two homotopies
$\Psi_1$ and $ \Psi_2$ such that
during $\Psi_1$ we only have self-crossing changes on $L_1$, while during $ \Psi_2$ we only have crossing changes
between $L_1$ and the rest of the components. The argument of Case 1 applies to $\Psi_1$ to conclude that  $X_{\Psi_1}=0$.
Since the restriction of $\Psi_2$ on $P'\times S^1$, where $P'=P\setminus P_1$, is constant;
it extends to a map $P'\times D^2 \longrightarrow M$. Then by Lemma \ref{disc1} we have
$X_{\Psi_2}=0$. 
\end{proof}

\subsection{The completion of the proof of Theorem \ref{integration2}}
Let $\Phi$ be a framing preserving loop in $M^L(P, M)$. Suppose that $\Phi| P_i\times S^1\longrightarrow M$ 
represents an essential torus for some  $i=1, \ldots, m$.  First suppose that some component, say $\Phi_i:= \Phi| P_i\times S^1\longrightarrow M$,
can be homotoped to lie on an embedded essential torus in $M$. Then a theorem of Nielsen
(\cite{hempel}, theorem 13.1) implies that after further homotopy,
we may assume that $\Phi_i$ is a covering map of an embedded torus.
It follows that the contribution of  $\Phi_i$  to $X_{\Phi}$ is zero.
Thus, for our purposes, we can assume that if
$\Phi_i$ induces an injection on $\pi_1$ then it cannot be homotoped to lie on an embedded torus.
Then by Lemma \ref{tori} we obtain $X_{\Phi}=0$.

As in the proof of Lemma \ref{tori} we may assume that $\Phi$ fixes all but one component of $L$; say $L_1$.
If $\Phi: P_1\times S^1\longrightarrow M$ is inessential  the argument in the proof of  Theorem  \ref{specialcase} applies to 
conclude that $X_{\Phi}=0$.
Assume that
$\Phi: P_1\times S^1\longrightarrow M$ is essential. Then $X_{\Phi}=0$ by Lemma  \ref{tori}.
 
\section{Kauffman power series } 
\subsection{Links in oriented $\Q$-homology 3-spheres} For framed links in $S^3$  the Kauffman polynomial is equivalent to a sequence of 1-variable Laurent polynomials ${\{R_n=R_n(t)\}}_{n\in \bf Z}$ determined by relations:
$$R_n(U)=1$$
$$R_n(L_r)=t^{{{{-(n+1)}}}} R_n(L)$$
$$R_n(L_l)=t^{{{{(n+1)}}}} R_n(L)$$
$$R_n({L_{+}})-R_n({L_{-}})=(t-t^{-{1}}) [ R_n(L_{o})-R_{n}(L_{\infty})]$$
where $L_{+}$, $L_{-}$, $L_o$, $L_{\infty}$ are as in Figure 1 and $L_r, L_l$ are as in Figure 2.
Notice that the initial value $R_n(U)=1$ is just a normalization. Any choice of the initial value
together with the rest of the relations will determine a unique $R_n$. Set $$u_n(t):= {{t^{n+1}-t^{-{(n+1)}}}\over {t-t^ {-1}}}+1.  \eqno(10)$$ By the relations above one obtains 
$R_n(L\sqcup U) = u_n(t)\  R_n(L) $,
where the link $L\sqcup U$ is obtained from $L$ by adding an unknotted and
unlinked component $U$. The coefficients of the power series $R_n(x)$ obtained from $R_n(t)$
by substituting $t=e^ x$ are invariants of finite type \cite{bn, bl}.
In the theorem below we reverse this procedure and guided by the axioms above we will 
construct power series invariants generalizing the
$R_n(x)$'s: Suppose that  $M$
is a ${\mathbb Q}$-homology sphere
with $\pi_2(M)=0$ and such that if $H_1(M)\neq 0$ then
$M$ is  atoroidal. For every $n\in \Z$ we will construct a sequence of framed link invariants
$\{ v_n^m |  m\in \N\}$
such that the formal power series
$$ R_{\{ M, n\}}= \sum_{m=0}^{\infty}v_n^m x^m$$
satisfy the axioms above under the change of variable $t={e}^x$: We will construct our invariants inductively (induction on $m$) by using Theorem \ref{integration}. Each $v_n^m$ 
is going
to be obtained by integrating a suitable singular link invariant
determined by the $v_n^j$'s with $j<m$. Although the resulting invariants will be invariants of unoriented  framed links,
for their construction we need to work with oriented links. The reason is that Theorem \ref{integration} applies to  oriented framed  links.
Recall that ${\mathcal L}$ (resp. ${\bar {\mathcal L}}$) denotes the set of isotopy classes of framed  oriented (resp. unoriented) links in $M$.
Also recall the set of oriented initial links
${\mathcal C}{\mathcal L}:= {\mathfrak o}^{-1}(\mathcal C \mathcal L^* \cup \{U\})$, defined in the beginning of subsection $\S 2.3$.
By Theorem \ref{integration} and its proof  the invariant $v_n^m$  is unique once the values on the set ${\mathcal C}{\mathcal L}$ are specified.

\begin{theorem}\label{main1}Assume that $M$ is a $\Q$-homology 3-sphere with $\pi_2(M)=0$
and such that if $H_1(M, {\mathbb Z})\neq 0$ then $M$ is  atoroidal. Fix $n\in \Z$. Given maps ${\mathcal V}_n^m: \mathcal C{\mathcal L}^*\cup\{ U\} \longrightarrow {\mathbb C}$, $m\in \N$,
there exists a unique sequence of complex valued 
link invariants $\{v_n^m | m\in \N\}$ with the following properties:
\begin{enumerate}
\item $v_n^m(CL)={\mathcal V}_n^m(\mathfrak o(CL))$ for all $CL\in {\mathcal C \mathcal L}$ and $m\in {\mathbb N}$.

\item  $v_n^m(L)=v_n^m({\mathfrak o}(L))$ for all $L\in {\mathcal L}$ and $m\in {\mathbb N}$. Thus the values of the invariants are independent of the link orientation.
\item  If we define a formal power series
$$R_n:= R_{\{ M, n\}}(L)= \sum_{m=0}^{\infty}v_n^m(L)x^m$$ then we have
$$R_n(U)=1\eqno (11)$$
$$R_n(L_r)=t^{{{{-(n+1)}}}} R_n(L)\eqno (12)$$
$$R_n(L_l)=t^{{{{(n+1)}}}}R_n(L)\eqno (13)$$
$$R_n({L_{+}})-R_n({L_{-}})=(t-t^{-{1}}) [ R_n(L_{o})-R_{n}(L_{\infty})]
\eqno (14)$$
where ${\displaystyle t:={e}^x=1+x+{{x^2}\over {2}}+  \dots}$.
\end{enumerate}
\end{theorem}
\begin{proof}  Define
$v_n^m(CL)={\mathcal V}_n^m(\mathfrak o(CL))$ for all $CL\in {\mathcal C \mathcal L}$
and $m\in {\mathbb N}$.
Now  we can form the power series
$R_n(CL)$. 
Guided by (12)-(13) we define
$$R_n(CL _r)=t^{{{{-(n+1)}}}} R_n(CL) \  {\rm and}\   R_n(CL_l)=t^{{{{(n+1)}}}}R_n(CL).\eqno(15)$$
Now guided by these we can define the values of $R_n$
on all framed links whose underlying unframed isotopy class is $CL$.
To explain this suppose that  $CL$ has $s$ components.
Let $CL({\bf f})$ be a framed link in the same (unframed) isotopy class with $CL$  
with framing unordered  sequence ${\bf f}$  (see Definition \ref{total} and preceding discussion).
Then define
$$R_n(CL({\bf f}))=t^{(n+1)\tau}R_n(CL), $$
where ${\tau}:={\tau}(CL({\bf f}))$ is the total framing of $CL({\bf f})$.
Using (14)-(15), and inducting on $k$, we can check that
$$ R_{n}(CL\sqcup U^k) =[ u_n(t)]^{k-1}\  R_n(CL)\eqno(16)$$
where $u_n(t)$ is given by (10). Now $R_n$ has been defined on all framed links in the unframed isotopy classes of the links in $\mathcal C \mathcal L$.

To continue for every  $L({\bf f})\in {\mathcal L}$ with framing sequence ${\bf f}$  we define 
$$v_n^0(L({\bf{f}}))=v_n^0(CL({\bf f})),$$
where
$CL$ is the initial link homotopic to $L$.
Inductively, suppose that the invariants
$v_n^0, v_n^1, \dots, v_{n}^{m-1}$
have been defined such that if we let
$$R_{ n}^{(m-1)}(L):= \sum_{i=1}^{m-1}v_n^i(L)x^i,$$
then we have
$$R_n^{(m-1)}(L_r)=t^{{{{-(n+1)}}}} R^{(m-1)}_n(L)\  {\rm mod}\  x^m  \eqno (17)$$
$$R_n^{(m-1)}(L_l)=t^{{{{(n+1)}}}}R_n^{(m-1)}(L)\  {\rm mod}\  x^m \eqno (18)$$
$$ R_{n}^{(m-1)}(L\sqcup U) = u_n(t)\  R_n(L) \ {\rm mod}\   x^m \eqno (19)$$
and
$$
R_n ^{(m-1)}({L_{+}}) 
-R_n^{(m-1)}({L_{-}})
= (t-t^{-1}) [ R_n^
{(m-1)}(L_{o})-R_n^
{(m-1)}(L_{\infty})]\ {\rm mod}\  x^m.$$
Furthermore, suppose that these invariants do not depend on the orientation of the links.
The last equation   leads us to define
$$
R_n^{(m)}({L_{\times}}):=
(t-t^{-{1}})[R_n^{(m-1)}(L_{o})-R_n^
{(m-1)}(L_{\infty})]\   {\rm mod} \ x^{m+1}   \eqno(20)
$$

We want to define the invariant $v_m^n$: Recall that it is already defined on the initial links.
Next we examine the right hand side of (20).
It  is a polynomial of degree $m$ such that  the coefficient of $x^m$ comes from
$$(t-t^{-{1}})[R_n^{(m-1)}(L_{o})-R_n^
{(m-1)}(L_{\infty})].$$
The expression above has no constant term and thus the coefficient
of $x^m$ depends on  the inductively well defined invariants
 $v_n^i$, $i=1$, $2$, $\ldots$,$m-1$.
Thus the coefficient of $x^m$ in (20) is  a ``new"  framed singular link invariant.
We are going to prove that it is derived from a framed  link invariant
by using Theorem \ref{integration}. For that we need to check  that  condition (3) in the statement \ref{integration} is satisfied. It is enough to check it
modulo $x^{m+1}$. In what follows the symbol ``$\equiv$" will denote calculation
modulo $x^{m+1}$.

Let $L_{\times +}$ and $L_{ \times- } \in {\mathcal L}^{(1)}$
be two singular framed links as in the left hand side of (3) in the statement \ref{integration}.  From  (20) we have
{
\begin{eqnarray*}
&&{R}_n^{(m)}(L_{{\times }+})-
R_n^{(m)}(L_{{\times }-}) \equiv\cr
&\equiv&
(t-t^{-{1}})\big[R_n^{(m-1)}(L_{o +})-R_n^{(m-1)}(L_{ \infty +})\big]-\cr
&-&(t-t^{-{1}})\ \big[R_n^{(m-1)}(L_{o -})-R_n^{(m-1)}(L_{\infty -}) \big]\equiv \cr
&\equiv& 
(t-t^{-{1}})\big[R_n^{(m-1)}(L_{o +})-R_n^{(m-1)}(L_{o -}) \big] -\cr
&-&(t-t^{-{1}})\big[R_n^{(m-1)}(L_{ \infty +})-R_n^{(m-1)}(L_{\infty -})\big]\equiv \cr
&\equiv&
 {(t-t^{-1})}^2 
\big[R_n^{(m-1)}(L_{oo})-R_n^{(m-1)}(L_{o \infty})\big] -\cr
&-&{
(t-t^{-{1}})}^2\big[ R_n^{(m-1)}(L_{\infty o})-R_n^{(m-1)}(L_{\infty \infty })
\big]\equiv \cr
&\equiv&
{(t-t^{-1})}^2 
\big[R_n^{(m-1)}(L_{oo})+R_n^{(m-1)}(L_{\infty  \infty})\big] -\cr
&-&{
(t-t^{-{1}})}^2\big[ R_n^{(m-1)}(L_{\infty o})+R_n^{(m-1)}(L_{o \infty })
\big].
\end{eqnarray*}
 }
Since the result is symmetric with respect to the two double points
we deduce that
$$R_n^{(m)}(L_{{\times}+})-R_n^{(m)}(L_{{\times}-})\equiv
R_n^{(m)}(L_{+{\times}})-R_n^{(m)}(L_{-{\times}}).$$


Thus the framed singular link invariant defined above
is induced by a framed  link invariant.
Recall that  we have already defined the 
values of $v_n^m$ on all framed links with unframed underlying isotopy classes in $\mathcal C \mathcal L$.
Using this values we can define a link invariant $v_n^m$ for all links in $\mathcal L$ such that if we let
$$R_{ n}^{(m)}(L)= \sum_{i=1}^{m}v_n^m(L)x^i$$
we have
$$
R_n ^{(m)}({L_{+}}) 
-R_n^{(m)}({L_{-}})= R_n^{(m)}({L_{\times}}), \eqno(21)$$
for every $L_{\times} \in {\mathcal L}^{(1)}$.
Now it is  a straightforward calculation to check that the inductive hypotheses hold 
mod $x^{m+1}$. For example let us check (18);  the others are similar. Consider a framed link $L_r$.
Keeping the kink intact in a small 3-ball, make a sequence of crossing changes
to transform $L_l$ to an initial link say $CL_l$. Over all such sequences of crossing changes, and initial links $CL_l$,
choose one that minimizes the number of the required crossing changes. Suppose, without loss of generality,
that the first crossing to be changed in that sequence is a positive crossing.
By (20) and (21) we have
$$
R_n ^{(m)}({L_{l+}}) \equiv
R_n^{(m)}({L_{l-}})
+ (t-t^{-1}) [ R_n^
{(m-1)}(L_{lo})-R_n^
{(m-1)}(L_{l\infty})]\ {\rm mod}\  x^{m+1}.$$
By (15) and induction on the number of crossing changes needed to go from $L_{l+}$ to $CL_l$
we can assume that
$$R_n^{(m)}(L_{l-})\equiv t^{{{{(n+1)}}}}R_n^{(m)}(L)$$
By (18) we have

$$R_n^{(m-1)}(L_{lo})=t^{{{{(n+1)}}}}R_n^{(m-1)}(L_o)\  {\rm mod}\  x^m, $$
and 
$$R_n^{(m-1)}(L_{l\infty})=t^{{{{(n+1)}}}}R_n^{(m-1)}(L_{\infty} )\ {\rm mod}\  x^m.$$
Combining the last four equations we have
$$
R_n ^{(m)}({L_{l+}})  \equiv t^{{{{(n+1)}}}} R_n ^{(m)}({L_{+}}),$$
as desired.
To finish the proof we need to show that $v_n^m$ is independent of the link orientation. Inductively, we assume that
$v_n^0, v_n^1, \dots, v_{n}^{m-1}$
are uniquely determined  by their values on ${\mathcal CL}$
and independent of the (singular) link orientation. We have that
$$v_n^m(L)=v_n^m(CL)+ \sum_{i=1}^s {\pm} v_n^m (L_i)$$
where $L_1, \dots, L_s\in {\mathcal L}^{(1)}$ are singular links appearing in a homotopy from $L$ to $CL$, where
$CL$ is the representative of $L$ in ${\mathcal CL}$ (compare relation (5)). Recall that we defined
$v_n^m(CL)$ to be independent of the orientation for $CL$.
The proof of Theorem \ref{integration} establishes that $v_n^m(L)$ doesn't depend on the homotopy from $L$ to $CL$ chosen.
By induction
 $v_n^0, v_n^1, \dots, v_{n}^{m-1}$ do not depend on orientations.
 It follows that $v_n^m(L)$ is unique once
 $v_n^m(CL)$ is chosen and independent on link orientation.
 \end{proof}
\vskip 0.04in

Theorem \ref{main2} stated in the Introduction is obtained from Theorem  \ref{main1} if we set $z:=it -(it)^{-1}=ie^x+ie^{-x}$ and $a:=i e^{y}$,
where $y=(n+1)x$. 
Now we derive Theorem \ref{main} stated in the Introduction.

\begin{proof}
The elements in the  set  $ \mathcal C \mathcal L^* \cup \{U\}$ are in one-to-one correspondence with a basis of $S({\hat \pi})$.
An element   $R\in {\mathfrak F}^{*}(M)$ gives rise to one in $S^{*}({\hat \pi})$ by restriction on the set 
$\mathcal C \mathcal L^* \cup \{U\}$. Thus one direction of the theorem follows. For the other direction, we note
that an element in  $S^{*}({\hat \pi})$  defines a map
${\mathcal R}_M: \mathcal C \mathcal L^* \cup \{U\} \rightarrow\hat \Lambda$. Then by Theorem
\ref{main2} there is a unique  map
 $R_M: {\bar {\mathcal L}}\rightarrow\hat \Lambda$ with properties  (1)-(3). These properties  guarantee that $R_M$ factors through the Kauffman module 
 ${\mathfrak F}(M)$ to give an element in ${\mathfrak F}^{*}(M)$
 (see Definition \ref{skein}).
\end{proof}
\vskip 0.04in

\subsection{Links in $ S^3$}Links in $S^3$ are studied via projections on a sphere $S^2 \subset S^3$.
Let $U^m$ denote the standard $m$-component unlink
and $U^m({\bf f})$ denote the one with framing $\bf f$. Every $m$-component  link projection $L\subset S^2$
is transformed to a framed unlink  by finitely many crossing changes and regular isotopy moves on
$S^2$ (i.e. isotopy using the  Reidemeister moves of type II and III only).
For a link  projection $L\subset S^2$, we define a complexity
$$s(L):=(u(L), \  c(L)), $$
as follows: $c(L)$ is the number of crossings of $L$
and
$u(L)$ is the number of admissible  crossing changes required to transform $L$ into  a diagram of a framed unlink that
that admits a type I
Reidemeister move that reduces its crossing number. 
We order the complexities lexicographically.
Let $R:=R_{S^3} : {\mathcal L}\rightarrow\hat \Lambda$ be a map  constructed as in  Theorem \ref{main2}
and recall that  $\Lambda:= {\mathbb C}[a^{\pm 1}, z^{\pm 1}]$. Note that the complexity $s(L)$ defined above has
the properties that $s(L_r), s(L_l)>s(L)$.

\begin{prop} \label{poly} Define $R(U({\bf f}))= a^{-{\tau}} (a+a^{-1}) z^{-1}+1$,
 where ${\tau}:= \sum_{i=1}^m{\bf f_i}$. Then,  $R(L)\in\Lambda$  for every link. In fact, $R(L)$ is the two variable Kauffman polynomial.
\end{prop}
\begin{proof}
Given a diagram $L$  first perform  all type I Reidemeister moves
that reduce the number of crossings of $L$. If there are no such moves and $L$ is non-trivial 
 then there is a crossing change such that
three of the  terms  $s(L_-), s(L_o), s(L_{\infty}), s(L_+)$ 
are strictly less that
the remaining fourth one.
Thus the skein relations
$$ R(L_+)-R(L_-)=z\big[ R(L_o)-R(L_\infty)\big],$$
$$R(L_r)=aR(L)\ \  {\rm and }\ \  R(L_l)=a^{-1}R(L)$$
allow us to write the invariant $R(L)$  of every link $L$ as a finite sum of 
the invariants of links of strictly less complexity than $s(L)$ and with coefficients in $\Lambda$. 
The result follows by induction on $s(L)$ and the observation  that
$R(U({\bf f}))\in \Lambda$. The last claim in the statement of the proposition follows by the uniqueness properties of $R$.
\end{proof}

\end{document}